\documentclass{IEEEtran}
\usepackage{cite}
\usepackage{amsmath,amssymb,amsfonts}
\usepackage{algorithmic,color}
\usepackage{epstopdf}
\usepackage{graphicx}
\usepackage{textcomp}
\usepackage{mathrsfs}
\usepackage{booktabs}
\def\BibTeX{{\rm B\kern-.05em{\sc i\kern-.025em b}\kern-.08em
    T\kern-.1667em\lower.7ex\hbox{E}\kern-.125emX}}

\newtheorem{theorem}{Theorem}

\newtheorem{remark}{Remark}
\newtheorem{lemma}{Lemma}

\usepackage{tikz}
\usetikzlibrary{shapes,arrows,positioning}
\usepackage{algorithm}
\usepackage{algorithmic}
\allowdisplaybreaks
\begin{document}

\title{Distributed Load Frequency Control of Multi-Area Smart Grid}

\author{Wenjing Yang, Zhaorong Zhang, Xun Li, Juanjuan~Xu
\thanks{This work was supported by the National Natural Science Foundation of China under Grants 62573262, 62503289 and the Natural Science Foundation of Shandong Province under Grants ZR2021JQ24. \textit{(Corresponding author: Juanjuan~Xu.)}}
\thanks{Wenjing Yang and Juanjuan~Xu are with School of Control Science and Engineering, Shandong University, Jinan, Shandong, P.R. China 250061.
        {\tt\small yangwenjing1024@163.com, juanjuanxu@sdu.edu.cn}}%
\thanks{Zhaorong Zhang is with School of Computer Science and Technology, Shandong University, Qingdao, Shandong, P.R. China 266237.
        {\tt\small zhangzr@sdu.edu.cn}}%
\thanks{Xun Li is with the Department of Applied Mathematics, Hong Kong Polytechnic University, Hong Kong, China.
        {\tt\small li.xun@polyu.edu.hk}}%
}

\maketitle

\begin{abstract}
In this paper, we investigate the distributed load frequency control problem in a multi-area smart grid under external load disturbances and measurement noise. The novelty lies in that the information privacy is fully taken into account, that is, the internal structural parameters and operational states of each area are not shared with non-neighboring areas, which makes traditional distributed optimal control methods ineffective. The main contribution is to propose a distributed algorithm for the global optimal power regulation command under information privacy constraints via distributed approximation of the control Riccati equation, the estimation Riccati equation, and the state estimation. Simulation results show that the proposed algorithm can approximate the performance of centralized optimal control, and the performance index under the proposed distributed controller is smaller than that under the commonly used distributed control.
\end{abstract}

\begin{IEEEkeywords}
Load frequency control, Multi-area smart grid, Distributed algorithm.
\end{IEEEkeywords}

%
\IEEEpeerreviewmaketitle

\section{Introduction}
With the rapid advancement of smart components and network communication technologies, the multi-area smart grid has attracted extensive attention due to its interconnected and intelligent characteristics \cite{1}. Such systems consist of multiple interconnected control areas that utilize communication networks to enable real-time information exchange across areas and rely on physical tie-lines to facilitate power exchange between areas. As this cyber-physical integration enhances the flexibility of energy dispatch and improves power supply reliability, it also introduces increased complexity in maintaining full-system operational stability. In particular, dynamic disturbances originating in one area may propagate to others through both physical and cyber pathways, potentially compromising the frequency stability and power balance of the overall grid \cite{2}.

Under such circumstances, as a key means of ensuring frequency stability and power balance in power systems, load frequency control (LFC) plays a central role in multi-area power systems \cite{3}. Traditionally, frequency stability in each area has been achieved by a centralized decision maker \cite{4}, i.e., so-called centralized control. However, the centralized control method requires a central processor to collect all the system's information in real time, which presents issues such as high communication demands and poor reliability. Consequently, distributed control methods have been extensively studied\cite{5,6}, where different areas make decisions using their own and neighbouring areas' information to achieve coordinated area control. For example, the distributed MPC schemes were investigated in \cite{7,8} to achieve distributed intelligent regulation of both system voltage and frequency. \cite{9} studied an LFC intelligent control for communication topology changes based on multi-agent system technology. \cite{10} studied a distributed control strategy based on an event-triggered mechanism to conserve communication resources. A distributed consensus algorithm based on an observer was discussed in \cite{11} to achieve actual frequency synchronization and power-sharing under unknown time-varying power demand. \cite{12} investigated a distributed linear quadratic multi-area LFC control method, which proves that globally optimal distributed controllers depend on the structure of the penalty matrix. \cite{13} investigated a distributed optimal frequency control algorithm for generators and controllable loads in a power transmission network. A distributed robust adaptive control algorithm was investigated in \cite{14} to address the LFC problem in power systems with dynamic net loads. \cite{15} investigated a distributed LQR control method for large-scale multi-area power systems with complete communication topology, which can approximate a centralized optimal controller.

Furthermore, due to external load disturbances, measurement noise, and uncertainties in system dynamics, the state variables in LFC systems are often not directly accessible. Therefore, designing a state observer based on measurable signals to achieve accurate estimation of the system state has become necessary. For instance, \cite{16} addressed a distributed robust secondary voltage and frequency control based on an extended state observer for autonomous micro-grids with inverter-based distributed generators. For multi-area power systems under cyber-physical attacks, \cite{17} presented an event-triggered LFC scheme based on distributed observers. An optimal state feedback control strategy is employed in \cite{18} on the basis of state estimation to perform load-frequency control for each region separately. \cite{19} studied two attack strategies against distributed state estimation in smart grids and gave the corresponding attack methods. An observer-based finite-time fuzzy LFC strategy was proposed in \cite{20} for multi-area nonlinear power systems with cyber attacks and input delays. \cite{21} proposed a decentralized LFC method relying on dynamic state estimation. \cite{22} investigated a robust LFC scheme that integrates second-order sliding mode control and an extended disturbance observer for multi-area power systems. It is worth noting that despite plenty of important advances in distributed control methods for LFC, most of the existing results are suboptimal unless the topological graph is complete. Such idealized topologies are difficult to achieve in practical power systems due to geographical limitations and cost-effectiveness. Therefore, there is a need for designing distributed optimal power regulation commands for each control area that can approximate the optimal performance of a centralized controller by using its own and neighboring areas' information.

This paper investigates the distributed LFC problem in a multi-area smart grid with external load disturbances and measurement noise. In contrast to traditional distributed optimal control methods \cite{12}, \cite{15} that require the sharing of internal structural parameters, including generator parameters, governor parameters, and tie-line coefficients of all areas, the information of each area in this paper is private and can only be shared with neighboring areas. The main contribution is designing a distributed algorithm that can achieve global optimality under information privacy constraints. The proposed algorithm consists of the distributed solving of the control Riccati equation, the estimation Riccati equation, and the state estimation. Simulation results demonstrate that the control performance of the developed method can approximate the centralized optimal control performance. In particular, the resulting performance index under the developed distributed optimal power regulation commands is smaller than that obtained by common distributed control method.

This paper is structured as follows. Section II presents the LFC model of the multi-area smart grid. Section III introduces the communication topology and the centralized optimal power regulation strategy, which serves as a performance benchmark. Section IV develops the distributed iterative algorithm that achieves global optimal control under information privacy constraints. Section V provides numerical simulations to validate the effectiveness of the proposed method. Finally, Section VI concludes the paper.

\textbf{Notation:} $\mathbb{R}^n$ represents the set of all $n$-dimensional real vectors; $\mathbb{R}^{n\times n}$ stands for the set of $n\times n$-dimensional real matrices; $P > 0$ ($P \geq 0$, respectively) means that
$P$ is a positive definite (positive semi-definite, respectively) matrix. $A'$ denotes the transpose operation of a vector or matrix $A$; and $I$ is the identity matrix with appropriately compatible dimensions.

\section{LFC model of multi-area smart grid}

This paper studies the distributed LFC problem of a multi-area smart grid consisting of $N$ areas. As shown in Fig.~\ref{fig1}, each area exchanges power with its neighboring areas $j \in \mathcal{N}_i^p$ through tie-lines and shares frequency information with neighbouring areas $j \in \mathcal{N}_i^c$ via a communication network. The control structure of the $i$-th area is depicted in Fig.~\ref{fig2}. According to \cite{5}, the dynamics of the LFC system can be described by the following differential equations:
\begin{align}
&\frac{d}{dt} \Delta f_i = \frac{1}{J_i} ( \Delta P_{m,i} - G_i \Delta f_i - \Delta P_{tie,i} - \Delta P_{d,i} ), \label{pr1} \\
&\frac{d}{dt} \Delta P_{m,i} = \frac{1}{T_{t,i}} ( \Delta P_{v,i} - \Delta P_{m,i} ), \label{pr2} \\
&\frac{d}{dt} \Delta P_{v,i} = \frac{1}{T_{g,i}}(\Delta P_{r,i} -\frac{1}{W_i} \Delta f_i - \Delta P_{v,i} ), \label{pr3} \\
&\frac{d}{dt} \Delta P_{tie,i} = \sum_{j \in \mathcal{N}_i^{p}}2 \pi  T_{ij} ( \Delta f_i - \Delta f_j ),\label{pr4}\\
&ACE_i = \beta_i \Delta f_i + \Delta P_{tie,i}, \label{pr5}
\end{align}
where $i$ is the control area, $\Delta f_i$ is the frequency deviation, $\Delta P_{m,i}$ is the mechanical power deviation of the generator, $\Delta P_{tie,i}$ is the tie line power deviation, $\Delta P_{d,i}$ is the external load disturbance, $\Delta {P}_{v,i}$ is the governor valve position deviation, $\Delta P_{r,i}$ is the power regulation command, which is the control input variable to be designed, $\Delta f_j$ is the frequency deviation of the interconnected area $j$. $T_{t,i}, T_{g,i}$ are the turbine and governor time constants, respectively. $G_i$ is the generator damping coefficient. $J_i$ is the generator rotational inertia. $W_{i}$ is the governor coefficient. $ACE_i$ is the area control error (ACE), and $\beta_{i}$ is the frequency deviation factor, $T_{ij}$ is the interconnection line coefficients between $i$ and $j$.
\begin{figure}
\centering
  \includegraphics[width=6.8cm,height=4.2cm]{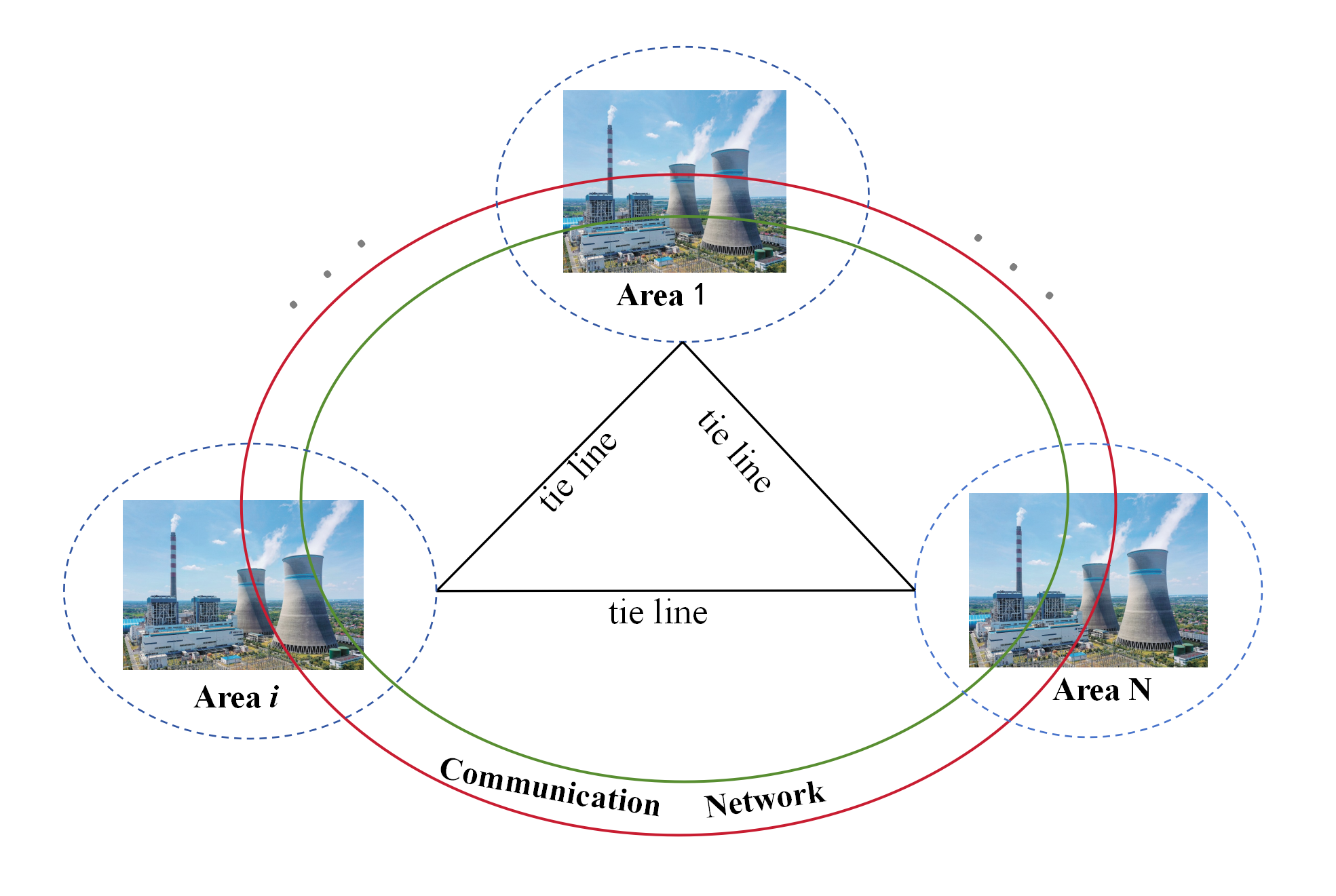}\\
  \caption{Multi-area smart grid structure.}\label{fig1}
\end{figure}
\begin{figure}
\centering
  \includegraphics[width=8.3cm,height=4.4cm]{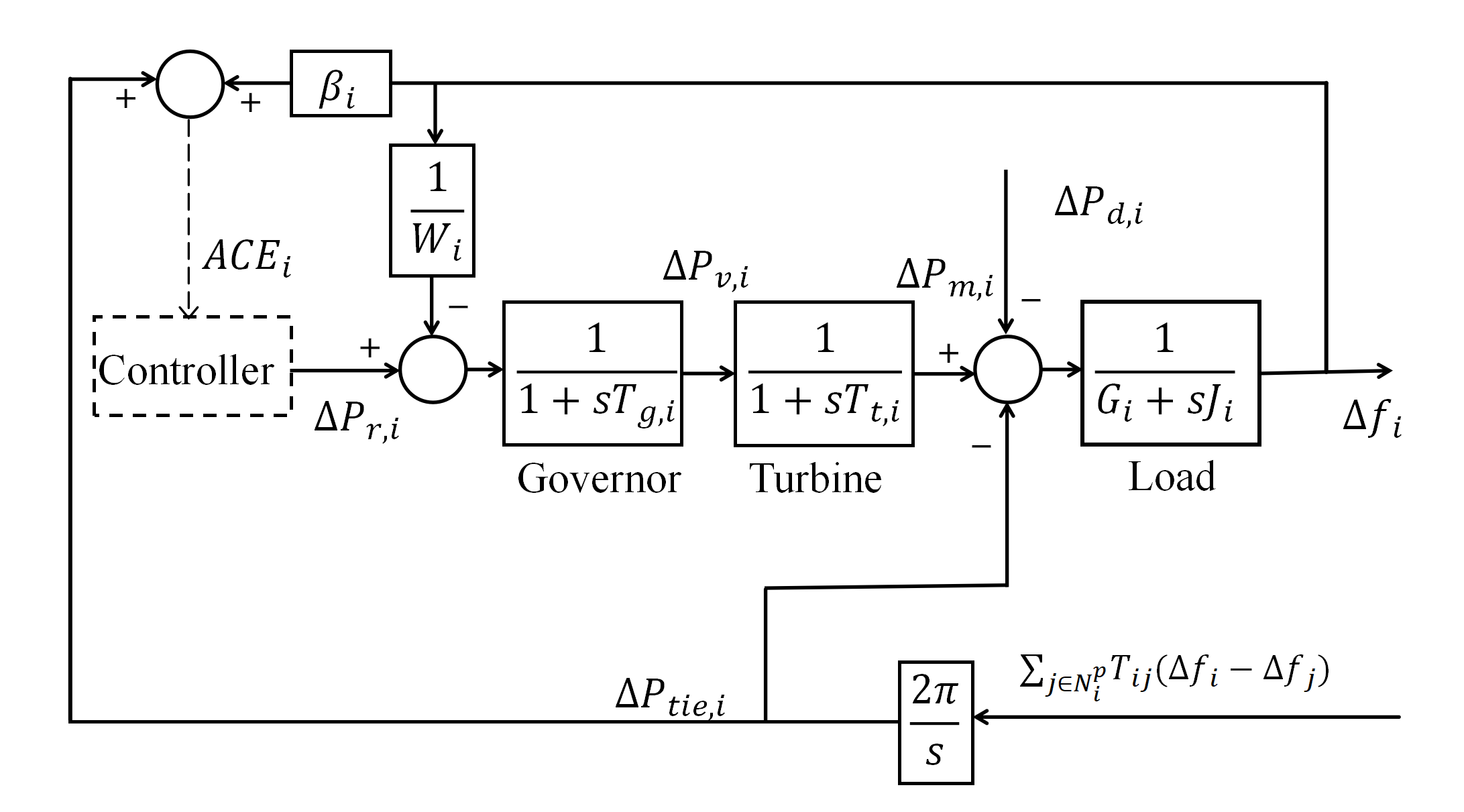}\\
  \caption{The structure of $i$th area.}\label{fig2}
\end{figure}

It is noted that the aforementioned dynamic model (\ref{pr1})-(\ref{pr5}) mathematically corresponds to a deterministic system, which fails to account for the effects of modeling uncertainties, external load fluctuations, and frequency measurement noise on each control area in practical grid operation. To enhance the control performance and practical applicability of the system, noise terms are introduced to characterize the aforementioned uncertain factors in this paper. In detail, by defining
\begin{align}
x_i(t) &= [\Delta f_i~\Delta P_{m,i}~\Delta P_{v,i}~\Delta P_{tie,i}]', \label{pr20}\\
u_i(t) &= \Delta P_{r,i}, ~~y_i(t) = ACE_i,\label{pr21}
\end{align}
the stochastic LFC model is formulated as:
\begin{align}
 \hspace{-1em}\dot{x}_i(t) &= A_{ii}x_{i}(t) + \sum_{j \in \mathcal{N}_{i}^{p}} A_{ij}x_{j}(t) + B_{ii}u_{i}(t) + w_i(t), \label{pr6}\\
 \hspace{-1em}y_i(t) &= C_{ii}x_{i}(t) + v_i(t), \label{pr7}
\end{align}
where
\begin{align}
&A_{ii}=
\begin{bmatrix}
-\dfrac{G_i}{J_i} & \dfrac{1}{J_i} & 0 & -\dfrac{1}{J_i} \\
0 & -\dfrac{1}{T_{t,i}} & \dfrac{1}{T_{t,i}} & 0 \\
-\dfrac{1}{W_iT_{g,i}} & 0 & -\dfrac{1}{T_{g,i}} & 0 \\
2\pi \sum_{j \in \mathcal{N}_i^{p}} T_{ij} & 0 & 0 & 0
\end{bmatrix}, \label{pr8}\\
&A_{ij}=
\begin{bmatrix}
0 & 0 & 0 & 0 \\
0 & 0 & 0 & 0 \\
0 & 0 & 0 & 0 \\
-2\pi T_{ij} & 0 & 0 & 0
\end{bmatrix}, B_{ii}=
\begin{bmatrix}
0 \\
0 \\
\dfrac{1}{T_{g,i}} \\
0
\end{bmatrix},\label{pr22}\\
&C_{ii} =
\begin{bmatrix}
\beta_i & 0 & 0 & 1
\end{bmatrix}.\label{pr23}
\end{align}
The term $w_i(t)$ represents the external load disturbance and modeling uncertainty, and $v_i(t)$ represents the measurement noise in frequency and tie-line power signals. Specifically, both $w_i(t)$ and $v_i(t)$ follow a Gaussian distribution with means $E_{wi}$, $E_{vi}$ and covariances $R_{wi}$, $R_{vi}$ respectively. The initial state $x_i(0)=x_{i0}$ is also a Gaussian random variable, with mean $\hat{x}_{i0}$ and covariance $R_{xi}$, and is independent of $w_i(t)$ and $v_i(t)$ over the interval $t\in [0, T]$.

To achieve frequency balance and frequency deviation, reduce control energy consumption and frequency deviation, we introduce the following linear quadratic performance index:
\begin{align}
J=E\{\int_{0}^T[x'(t)Qx(t)+u'(t)Ru(t)]dt+x'(T)Sx(T)\},\label{dc3}
\end{align}
where $x(t) = [x_1(t)~\cdots~x_N(t)]' \in \mathbb{R}^{4N}$, $u(t) = [u_1(t)~\cdots~u_N(t)]' \in \mathbb{R}^{N}$, $Q = \frac{1}{N}\sum_{i=1}^N Q_i$ with $Q = Q' \succeq 0$, and $R =\text{diag}\{R_1, R_2, \cdots, R_N\} \succ 0$, with compatible dimensions. In particular, the matrices $Q_i$ and $R_i$ are private to each area $i$.

It can be seen from the system model (\ref{pr6})-(\ref{pr7}) and performance index (\ref{dc3}) that each area $i$ has its own information characteristics, including $A_{ii},~ B_{ii},~ C_{ii},~ R_{wi},~ R_{vi},~ Q_i,~ R_i,~ x_i(t),~ y_i(t)$. Considering the information security and privacy protection requirements of multi-area smart grids, such information can only be shared among communication neighbors \cite{1,10}. To this end, this paper investigates distributed optimal power regulation strategies under the information privacy constraints. Mathematically, we define the set of private information for area $i$ as:
\begin{align}
\mathcal{F}_i(t) = \mathcal{W}_i(t) \cup \left( \bigcup_{j \in \mathcal{N}_i^c} \mathcal{W}_j(t) \right), \quad i = 1,2,\dots,N,
\end{align}
where $\mathcal{W}_i(t) = \{A_{ii}, ~\sum_{j \in \mathcal{N}_i^{p}} A_{ij},~ B_{ii},~ C_{ii},~ R_{wi}, ~R_{vi},,~Q_i,\\~R_i,~x_i(\tau),~ y_i(\tau), \tau \leq t\}$ denotes the private information of area $i$, and $\mathcal{N}_i^c$ represents the communication neighbor set of area $i$, whose rigorous definition will be given in the following Section.

Under the above information constraint, the problem investigated in this paper is formulated as follows.

\textbf{\textit{Problem}.} The objective of this paper is to design a distributed power regulation command such that the cost function defined in (\ref{dc3}) is minimized subject to (\ref{pr6})-(\ref{pr23}), where each area $i$ can only access its own and communication neighbours' internal structural parameters and operational states $\mathcal{F}_i(t)$.


\section{Preliminaries}
\subsection{Multi-Area Smart Grid Over Communication Networks}
The communication information exchange among control areas in a multi-area smart grid is modeled by an undirected communication network graph $\mathcal{G}^c= (\mathcal{N}, \mathcal{E}^c, \mathcal{A})$, where $\mathcal{N} = \{1, 2, \dots, N\}$ represents the collection of control areas, and $i$ is the $i$th area. The edge set $\mathcal{E}^c \subseteq \mathcal{N} \times \mathcal{N}$ characterizes the available communication transmission links, while $\mathcal{A} = [a_{ij}]_{N \times N}$ is the associated weighted adjacency matrix, with $a_{ij} > 0$ if $(j, i) \in \mathcal{E}^c$, and $a_{ij} = 0$ otherwise. Here, $(i, j) \in \mathcal{E}^c$ indicates that area $j$ can directly receive information from area $i$. For each area $i$, the set of its neighbors is defined as $\mathcal{N}_i^c = \{ j \mid (j, i) \in \mathcal{E}^c\}$. The Laplacian matrix of $\mathcal{G}^c$ is denoted by $\mathcal{L} = [l_{ij}]_{N \times N}$, where $l_{ii} = \sum_{j=1}^N a_{ij}$, and $l_{ij} = -a_{ij}$ for $i \neq j$. In particular, the physical tie-line connections are characterized by a separate graph $\mathcal{G}^p = (\mathcal{N}, \mathcal{E}^p, \mathcal{T})$. The edge set $\mathcal{E}^p$ characterizes the available physical tie-line transmission links.  And, $\mathcal{T} = [T_{ij}]_{N\times N}$ is the associated weighted adjacency matrix, with $T_{ij} > 0$ if a tie-line directly connects areas $i$ and $j$; otherwise $T_{ij} = 0$. The physical neighbors of area $i$ are $\mathcal{N}_i^p = \{ j \mid (j, i) \in \mathcal{E}^p\}$.
\subsection{Centralized Optimal Power Regulation Strategy}

To address the \textbf{\textit{Problem}} and provide a performance benchmark, this section first proposes a centralized optimal power control strategy. In this scenario, each area can share all areas' internal structural parameters and operational states, and the available information is mathematically represented as
\begin{align}
\mathcal{F}(t) = \bigcup_{i=1}^{N} \mathcal{W}_i(t).
\end{align}
Based on $\mathcal{F}(t)$, the global system matrices $A$, $B$ and $C$ for the overall multi-area system as follows:
\begin{align*}
A = \sum_{i=1}^N A_i, \quad B=\sum_{i=1}^N B_i, \quad C = \sum_{i=1}^N C_i,
\end{align*}
with $A_i, B_i, C_i$ being the local system matrices of area $i$, given by
\begin{align*}
A_i &= E_i' A_{ii} E_i + \sum_{j \in \mathcal{N}_i^{p}} E_i' A_{ij} E_j, \\
B_i &= E_i' B_{ii}e_i, ~
C_i = e_i' C_{ii}E_i, \\
E_i &= [0~\cdots~0~I~0~\cdots~0],\\
e_i &= [0~\cdots~0~1~0~\cdots~0].
\end{align*}
Furthermore, denote $R_x=\text{diag}\{R_{x1}, R_{x2},\ldots, R_{xN}\}$, $R_w=\text{diag}\{R_{w1}, R_{w2},\ldots, R_{wN}\}$, $R_v=\text{diag}\{R_{v1}, R_{v2}, \ldots, R_{vN}\}$. Then, the centralized optimal power regulation command is given below.

\begin{theorem}\label{the1} The centralized optimal power regulation command is given by:
\begin{align}
u^{*}(t)=-K(t)\hat{x}^{*}(t),\label{dc7}
\end{align}
where $\hat{x}^{*}(t)$ is the optimal state estimation based on all available ACE measurements across the entire smart grid, which satisfies:
\begin{align}
\dot{\hat{x}}^{*}(t) &= A\hat{x}^{*}(t)+Bu^{*}(t)+L(t)[y(t)-C\hat{x}^{*}(t)],\label{dc8}\\
\dot{\Sigma}(t) &= A\Sigma(t)+\Sigma(t)A'-L(t)C\Sigma(t)+R_{w},\label{dc5}\\
L(t) &= \Sigma(t)C'R_{v}^{-1},\label{dc6}
\end{align}
with $\Sigma(0)=R_{x}$, $\hat{x}^{*}(0)=\hat{x}(0)$, and the control gain $K(t)$ satisfies:
\begin{align}
K(t)&=R^{-1}B'P(t),\label{dc39}\\
\dot{P}(t)&=-A'P(t)-P(t)A-Q+P(t)BR^{-1}B'P(t),\label{dc4}
\end{align}
with $P(T)=S$.
\end{theorem}
\emph{Proof.} The derivation is standard and follows directly from the LQG framework established in Theorem 14.7 of \cite{23}. \hfill $\blacksquare$

\section{Distributed Optimal Power Regulation Strategy}

Recalling that each area $i$ can only access the internal structural parameters and operational states $\mathcal{F}_i(t)$ of itself and its communication neighbours, this renders the centralised optimal power control command in Theorem ~\ref{the1} inapplicable, and traditional distributed optimal control methods \cite{12}, \cite{15} are also ineffective.

In this section, we propose a distributed iterative algorithm based on $\mathcal{F}_i(t)$ to approximate the centralized optimal power regulation command (\ref{dc7}). From (\ref{dc7}), it is evident that the optimal power regulation command relies on three global quantities: the solution $P(t)$ of the control Riccati equation (\ref{dc4}), the solution \(\Sigma(t)\) of the estimation Riccati equation (\ref{dc5}), and the state estimate $\hat{x}^{*}(t)$ of (\ref{dc6}). Accordingly, the distributed solution scheme is decomposed into the distributed approximation of \(P(t)\), \(\Sigma(t)\), and $\hat{x}^{*}(t)$, respectively. The overall architecture is illustrated in Fig.~\ref{fig10}, with the detailed design presented as follows.
\begin{figure}[h]
\centering
  \includegraphics[width=8.6cm,height=4.5cm]{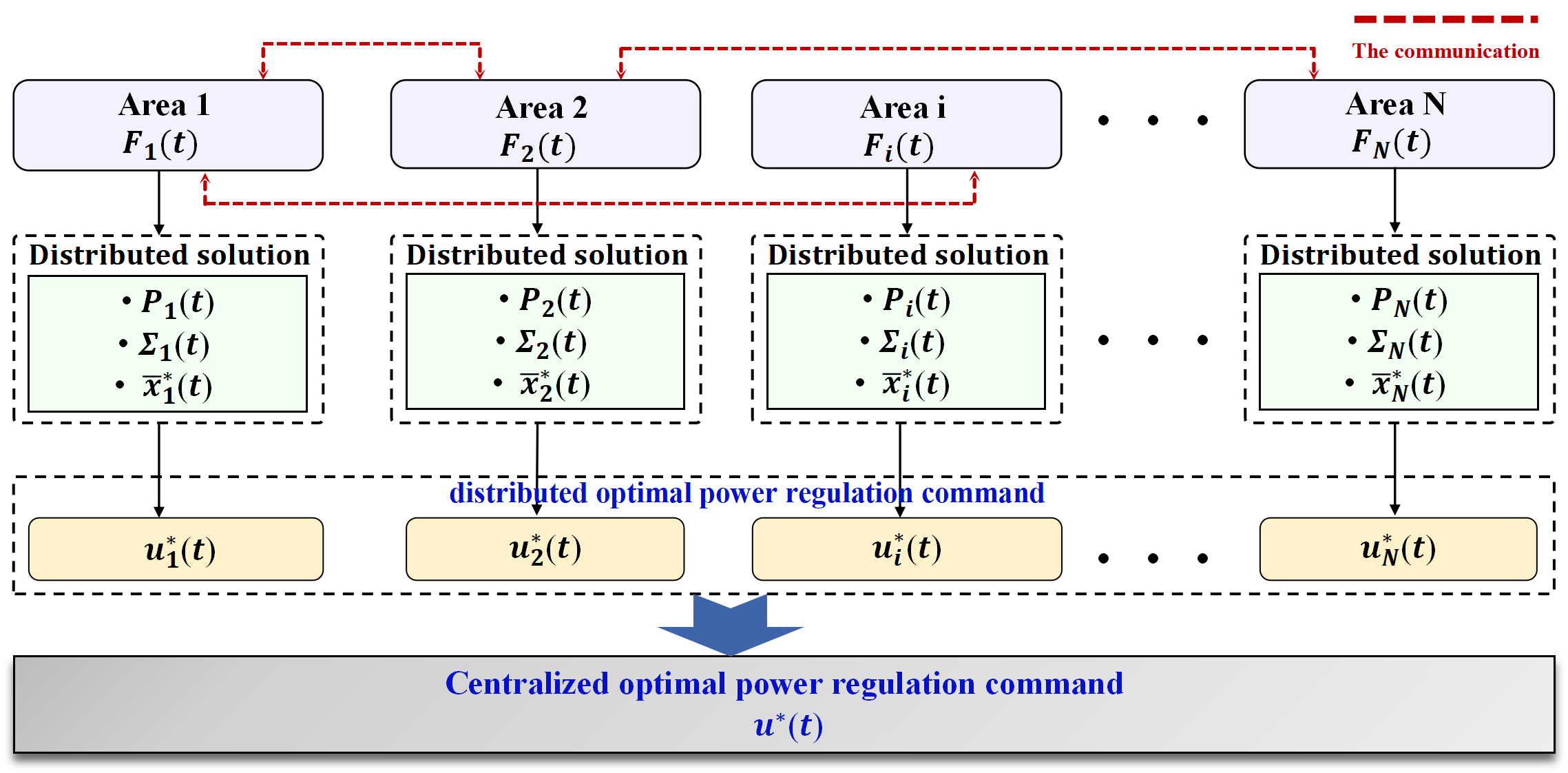}\\
  \caption{The overall distributed architecture.}\label{fig10}
\end{figure}

\begin{itemize}
\item\emph{\textbf{The distributed solution of \(P(t)\)}}
\end{itemize}

First, we present the distributed iterative scheme for \(P(t)\). It is worth noting that the control Riccati equation (\ref{dc4}) is essentially a nonlinear differential equation, which is already difficult to solve analytically. To this end, we propose an distributed iterative algorithm, enabling each control area to approximate \(P(t)\) using only its own and neighboring information, with the iterative algorithm is given as follows:
\begin{align}
X_{i,r}^{m}(t)=&X_{i,r-1}^{m}(t)+\alpha_r[\tilde{A}_{i}-\tilde{B}_{i}\tilde{R}_{i}
\tilde{B}_{i}'P_{i}^{m-1}(t)\nonumber\\
&\quad
-X_{i,r-1}^{m}(t)]
+\frac{1}{\delta}\sum_{j\in\mathcal{N}_i^c}[X_{j,r-1}^{m}(t)\nonumber\\
&\quad-X_{i,r-1}^{m}(t)],\label{dc20}\\
Y_{i,r}^{m}(t)=&Y_{i,r-1}^{m}(t)+\alpha_r[Q_{i}+P_{i}^{m-1}(t)
\tilde{B}_{i}\tilde{R}_{i}
\tilde{B}_{i}'
\nonumber\\
&\times P_{i}^{m-1}(t)-Y_{i,r-1}^{m}(t)]+\frac{1}{\delta}\nonumber\\
&\times\sum_{j\in\mathcal{N}_i^c}[Y_{j,r-1}^{m}(t) -Y_{i,r-1}^{m}(t)],\label{dc21}\\
\dot{P}_{i}^{m}(t)=&-X_{i,\infty}^{m}(t)'P_{i}^{m}(t)-P_{i}^{m}(t)
X_{i,\infty}^{m}(t)\nonumber\\
&
-Y_{i,\infty}^{m}(t),\label{dc18}
\end{align}
where \(r=1,2,\cdots\) and \(m=1,2,\cdots\) denote the inner and outer iteration indices, respectively. $\tilde{R}_{i}=\text{diag}\{0, \cdots, (1/N)R_i^{-1}, \cdots, 0\}$, $\tilde{A}_{i}=NA_{i}$, $\tilde{B}_{i}=NB_{i}$. $X_{i,\infty}^{m}(t) \triangleq \lim_{r\to\infty} X_{i,r}^{m}(t)$ and $Y_{i,\infty}^{m}(t) \triangleq \lim_{r\to\infty} Y_{i,r}^{m}(t)$ denote the limiting matrices of $X_{i,r}^{m}(t)$ and $Y_{i,r}^{m}(t)$ as the inner iteration $r$ converges, respectively. The initial conditions are
$X_{i,0}^{m}(t)=0, Y_{i,0}^{m}(t)=0$, $P_{i}^{0}(t)=0$, $P_{i}^{m}(T)=0$.

\emph{Practical implementation:} For each outer iteration $m$, each area executes the inner iterations (\ref{dc20})-(\ref{dc21}) using local and neighbor information until convergence within a prescribed accuracy $\sigma>0$, i.e., until $\|X_{i,r}^{m}(t)-X_{i,r-1}^{m}(t)\|<\sigma$ and $\|Y_{i,r}^{m}(t)-Y_{i,r-1}^{m}(t)\|<\sigma$. These converged values are then used in (\ref{dc18}) to solve for $P_{i}^{m}(t)$. The outer iteration continues until $P_{i}^{m}(t)$ converges, i.e., until $\|P_{i}^{m}(t)-P_{i}^{m-1}(t)\|<\sigma$, yielding the converged value $P_{i}^{m}(t)$ which approximates the global control Riccati matrix $P(t)$.

\begin{itemize}
\item\emph{\textbf{The distributed solution of $\Sigma(t)$}}
\end{itemize}

Secondly, by exploiting the duality between $\Sigma(t)$ and $P(t)$, a distributed iterative algorithm for the estimation Riccati $\Sigma_{i}(t)$ for area $i$ can be given in a similar manner:
\begin{align}
W_{i,r}^{m}(t)&=W_{i,r-1}^{m}(t)+\alpha_r[\tilde{A}_{i}-\Sigma_{i}^{m-1}(t)\tilde{C}_{i}'\tilde{R}_{vi}
\tilde{C}_{i}\nonumber\\
&\quad-W_{i,r-1}^{m}(t)]+\frac{1}{\delta}\sum_{j\in\mathcal{N}_i^c}[W_{j,r-1}^{m}(t)\nonumber\\
&\quad
-W_{i,r-1}^{m}(t)],\label{dc22}\\
G_{i,r}^{m}(t)&=G_{i,r-1}^{m}(t)+\alpha_r[\tilde{R}_{wi}+\Sigma_{i}^{m-1}(t)
\tilde{C}_{i}'\tilde{R}_{vi}\nonumber\\
&\quad \times
\tilde{C}_{i}\Sigma_{i}^{m-1}(t) -G_{i,r-1}^{m}(t)]\nonumber\\
&\quad+\frac{1}{\delta}\sum_{j\in\mathcal{N}_i^c}[G_{j,r-1}^{m}(t)-G_{i,r-1}^{m}(t)],\label{dc23}\\
\dot{\Sigma}_{i}^{m}(t)&=W_{i,\infty}^{m}(t)\Sigma_{i}^{m}(t)+\Sigma_{i}^{m}(t)
W_{i,\infty}^{m}(t)'\nonumber\\
&\quad+G_{i,\infty}^{m}(t),\label{dc19}
\end{align}
where $\tilde{C}_{i}=NC_{i}$, $\tilde{R}_{vi}=\text{diag}\{0, \cdots, (1/N)R_{vi}^{-1}, \cdots, 0\}$, $\tilde{R}_{wi}=\text{diag}\{0, \ldots,0,NR_{wi},0,\ldots,0\}$, $W_{i,\infty}^{m}(t) \triangleq \lim_{r\to\infty} W_{i,r}^{m}(t)$ and $G_{i,\infty}^{m}(t) \triangleq \lim_{r\to\infty} G_{i,r}^{m}(t)$ denote the limiting matrices as the inner iteration $r$ converges. The initial conditions are set as $W_{i,0}^{m}(t)=0$, $G_{i,0}^{m}(t)=0$, $\Sigma_{i}^{0}(t)=0$, $\Sigma_{i}^{m}(0)=\tilde{R}_{x}$.

\emph{Practical implementation:} For each outer iteration $m$, each area executes the inner iterations (\ref{dc22})-(\ref{dc23}) until convergence within a prescribed accuracy $\sigma>0$, i.e., until $\|W_{i,r}^{m}(t)-W_{i,r-1}^{m}(t)\|<\sigma$ and $\|G_{i,r}^{m}(t)-G_{i,r-1}^{m}(t)\|<\sigma$. These converged values are then used in (\ref{dc19}) to solve for $\Sigma_{i}^{m}(t)$. The outer iteration continues until $\Sigma_{i}^{m}(t)$ converges, i.e., until $\|\Sigma_{i}^{m}(t)-\Sigma_{i}^{m-1}(t)\|<\sigma$, yielding the converged value $\Sigma_{i}^{m}(t)$ which approximates the global estimation Riccati matrix $\Sigma(t)$.

\begin{itemize}
\item\emph{\textbf{The distributed solution of $\hat{x}^{*}(t)$}}
\end{itemize}

Finally, we compute the optimal state estimate $\hat{x}^{*}(t)$ in a distributed manner. From (\ref{dc8}), we have that
\begin{align}
\hat{x}^{*}(t)=\hat{\Omega}(t,0)x_{0}-\int_0^t\hat{\Omega}(t,\nu)L(\nu)y(\nu) d\nu,\label{dc30}
\end{align}
where $\hat{\Omega}(t,\nu)=\hat{\Omega}(t)\hat{\Omega}^{-1}(\nu)$ and $\Omega(t)$ satisfies the following differential equation:
\begin{align}
\dot{\hat{\Omega}}(t) = [A - BK(t) - L(t)C]\hat{\Omega}(t),   \label{dc31}
\end{align}
with the initial condition $\hat{\Omega}(0)=I$. It is evident that the optimal state estimation in (\ref{dc30}) relies on the associated state transition matrix $\hat{\Omega}(t, s)$. To compute this matrix in a distributed manner, each area first utilizes its own and neighboring areas' dynamic model parameters to approximate the global matrix $H(t)=\Sigma(t)C'R_{v}^{-1}C$ via the following distributed iterative algorithm:
\begin{align}
H_{i,\xi}(t) &= H_{i,\xi-1}(t) + \alpha_{\xi} \big[\Sigma_{i,\infty}(t)\tilde{C}_i'\tilde{R}_{vi}\tilde{C}_i-H_{i,\xi-1}(t) \big]\nonumber\\
&\quad + \frac{1}{\delta}\sum_{j \in \mathcal{N}_i^c} [H_{j,\xi-1}(t)- H_{i,\xi-1}(t)], \label{dc33}
\end{align}
where $\xi=1,2,\cdots$, $\Sigma_{i,\infty}(t)\triangleq \lim_{m\to\infty} \Sigma_{i}^{m}(t)$. The initial condition is $H_{i,0}(t)=0$. Subsequently, each control area $i$ can obtain the local state transition matrix $\hat{\Omega}_{i}(t)$ by solving the following differential equation:
\begin{align}
\dot{\hat{\Omega}}_{i}(t)=[X_{i,\infty}(t)
-H_{i,\infty}(t)]\hat{\Omega}_{i}(t),\label{dc35}
\end{align}
with $\hat{\Omega}_{i}(0)=I$, $X_{i,\infty}(t)\triangleq \lim_{m\to\infty} \lim_{r\to\infty} X_{i,r}^{m}(t)$ and $H_{i,\infty}(t)\triangleq \lim_{\xi\to\infty} H_{i,\xi}(t)$. Finally, each area utilizes its own and neighboring areas' ACE signals and initial state information $x_i(0)$ to approximate the global state estimate $\hat{x}^{*}(t)$ through the following distributed iterative algorithm:
\begin{align}
\bar{x}_{i,\tau}^{*}(t)& = \bar{x}_{i,\tau-1}^{*}(t)
+\alpha_\tau[\hat{\Omega}_{i}(t,0)\bar{x}_{i0}+\int_0^t
\hat{\Omega}_{i}(t,\nu)\nonumber\\
&\quad \times\Sigma_{i,\infty}(\nu)\tilde{C}_i'\tilde{R}_{vi}\bar{y}_{i}(\nu)d\nu
-\bar{x}_{i,\tau-1}^{*}(t)]\nonumber\\
&
\quad +\frac{1}{\delta}\sum_{j\in\mathcal{N}_i^c}[\bar{x}_{j,\tau-1}^{*}(t)-\bar{x}_{i,\tau-1}^{*}(t)], \label{dc36}
\end{align}
where $\tau=1,2,\cdots$, $\bar{y}_{i}(t)=[0~\cdots ~0~ Ny_{i}(t)~0~\cdots~0]'$. The initial condition is $\bar{x}^{*}_{i,0}(t)=0$, $\bar{x}_{i0}=[0,\cdots,0,~ Nx_{i0},~0,\cdots,0]'$.

\emph{Practical implementation:} Each area executes the iterations (\ref{dc33}) until convergence, i.e., until $\|H_{i,\xi}(t)-H_{i,\xi-1}(t)\|<\sigma$. These converged values are then used in (\ref{dc35}) to solve for local state transition matrix $\hat{\Omega}_{i}(t)$. Finally, each area performs (\ref{dc36}) until $E\|\bar{x}_{i,\tau}^{*}(t)-\bar{x}_{i,\tau-1}^{*}(t)\|<\sigma$ and the resulting converged value $\bar{x}_{i,\tau}^{*}(t)$ approximates the globally optimal state estimate $\hat{x}^{*}(t)$.

The following lemma verifies the convergence of the aforementioned iterative algorithms (\ref{dc20})-(\ref{dc36}), which serves as a fundamental theoretical prerequisite for the subsequent distributed optimal power regulation command design.

\begin{lemma}\label{le1}
Consider the distributed iterative algorithms (\ref{dc20})-(\ref{dc36}). Select a positive constant $\delta$ such that the matrix $I_N-\frac{1}{\delta}\mathcal{L}-\frac{1}{N}\mathbf{1}_N\mathbf{1}_N'$ is Hurwitz, and choose the step sizes $\alpha_r$ satisfying $\sum_{r=1}^\infty\alpha_r=\infty$ and $\sum_{r=1}^\infty\alpha_r^2<\infty$. Then, for all $t\in[0,T]$ and all areas $i\in\mathcal{N}$, the following convergences hold:
\begin{align*}
&\lim_{m\to\infty}\lim_{r\to\infty} X_{i,r}^{m}(t)=X(t),~ \lim_{m\to\infty}\lim_{r\to\infty} Y_{i,r}^{m}(t)=Y(t),\\
&\lim_{m\to\infty}\lim_{r\to\infty} W_{i,r}^{m}(t)=W(t),~\lim_{m\to\infty}\lim_{r\to\infty} G_{i,r}^{m}(t)=G(t),\\
&\lim_{m\to\infty} P_{i}^{m}(t)=P(t),\quad\quad\quad\lim_{m\to\infty} \Sigma_{i}^{m}(t)=\Sigma(t),\\
&\lim_{\xi\to\infty} H_{i,\xi}(t)=H(t),\quad\quad\quad
\lim_{\tau\to\infty} E\left\|\bar{x}_{i,\tau}^{*}(t)-\hat{x}^{*}(t)\right\|^2=0,
\end{align*}
where $X(t)=A-BR^{-1}B'P(t)$, $Y(t)=Q-P(t)BR^{-1}B'P(t)$, $W(t)=A-\Sigma(t)C'R_{v}^{-1}C$, $G(t)=R_{w}+\Sigma(t)C'R_{v}^{-1}C\Sigma(t)$, $H(t)=\Sigma(t)C'R_{v}^{-1}C$.
\end{lemma}

\emph{Proof.} The proof is similar to Lemma 2 in \cite{24} and is thus omitted here.\hfill $\blacksquare$

\begin{remark}
Based on the established convergence of the distributed iterative algorithms in Lemma 1, for any $\sigma>0$, there exist positive integer $\bar{M}$ such that for $m, \xi, \tau>\bar{M}-1$
\begin{align*}
&\|P_{i}^{m}(t)-P(t)\|^2<\sigma,\quad
\|\Sigma_{i}^{m}(t)-\Sigma(t)\|^2<\sigma,\\
&\|H_{i,\xi}(t)-H(t)\|^2<\sigma,\quad
E\|\bar{x}_{i,\tau}^{*}(t)-\hat{x}^{*}(t)\|^2<\sigma.
\end{align*}

Then, define the limiting values as follows:
\begin{align*}
&P_{i}^{*}(t)\triangleq P_{i}^{\bar{M}}(t),\quad
\Sigma_{i}^{*}(t)\triangleq \Sigma_{i}^{\bar{M}}(t),\\
&H_{i}^{*}(t)\triangleq H_{i,\bar{M}}(t),\quad
\bar{x}_{i}^{*}(t)\triangleq \bar{x}_{i,\bar{M}}^{*}(t),
\end{align*}
\end{remark}

With these convergence guarantees, each area can compute the distributed optimal power regulation command as follows.

\begin{theorem}\label{t2}
Under the same conditions as in Lemma \ref{le1}, the distributed optimal power regulation command for each area $i$ is
\begin{align}
u_i^{*}(t)=-[0~\cdots~0~R_i^{-1}B_i'~0~\cdots~0]P_{i}^{\ast}(t)\bar{x}_{i}^{\ast}(t),\label{dc40}
\end {align}
and $\bar{u}^*(t)=[u_1^*(t)~\cdots~u_N^*(t)]'$ approximates the centralized optimal power regulation command $u^*(t)$ in the mean-square sense, i.e. for any $\sigma>0$,
\begin{align}
E\|\bar{u}^{*}(t)-u^*(t)\|^2 < \sigma.\label{dc41}
\end{align}
\end{theorem}

\emph{Proof.} First, by noting that \begin{align}
\begin{bmatrix}
R_1^{-1}B_1' & 0 & \cdots & 0 \\
0 & R_2^{-1}B_2' & \cdots & 0 \\
\vdots & \vdots & \ddots & \vdots \\
0 & 0 & \cdots & R_N^{-1}B_N'
\end{bmatrix}
=R^{-1}B',\nonumber
\end{align}
we obtain
\begin{align}
&\bar{u}^*(t)-u^*(t)\nonumber\\
=&\begin{bmatrix}
(R_1^{-1}B_1'~\cdots~0)(P(t)\hat{x}^*(t)-P_{1}^{\ast}(t)\bar{x}_{1}^{\ast}(t))  \\
\vdots  \\
(0~\cdots~R_N^{-1}B_N')(P(t)\hat{x}^*(t)-P_{N}^{\ast}(t)\bar{x}_{N}^{\ast}(t))
\end{bmatrix}. \label{dc41e}
\end{align}

Next, by applying the Cauchy-Schwarz inequality, it yields
\begin{align}
&\mathbb{E}\|\bar{u}^*(t)-u^*(t)\|^2 \nonumber\\
~~\le& \sum_{i=1}^N \mathbb{E}\left\|\big(0~\cdots~R_i^{-1}B_i'~\cdots~0\big)\big(P(t)\hat{x}^*(t)-P_i^*(t)\bar{x}_i^*(t)\big)\right\|^2 \nonumber\\
\le& \sum_{i=1}^N C_i \mathbb{E}\left\|P(t)\hat{x}^*(t)-P_i^*(t)\bar{x}_i^*(t)\right\|^2.  \label{dc41f}
\end{align}
where $C_i = \|(0~\cdots~R_i^{-1}B_i'~\cdots~0)\|^2 = \|R_i^{-1}B_i'\|^2$ is a positive constant.

From the (\ref{pr6})-(\ref{pr7}), (\ref{dc8}) and (\ref{dc4}), there exists positive constants $C_x$, $C_P$ such that $\mathbb{E}\|\hat{x}^*(t)\|^2 \le C_x$, $\|P_i^*(t)\|^2\le C_P$ for all $t \in [0,T]$. Then, combining with Lemma \ref{le1} and Remark 1, we have
\begin{align}
&\mathbb{E}\left\|P(t)\hat{x}^*(t)-P_i^*(t)\bar{x}_i^*(t)\right\|^2 \nonumber\\
&\quad \le 2\mathbb{E}\left[\|P_i^*(t)\|^2\left\|\bar{x}_i^*(t)-\hat{x}^*(t)\right\|^2\right] \nonumber\\
&\qquad + 2\mathbb{E}\left[\left\|P_i^*(t)-P(t)\right\|^2\left\|\hat{x}^*(t)\right\|^2\right] \nonumber\\
&\quad \le 2C_P \mathbb{E}\left\|\bar{x}_i^*(t)-\hat{x}^*(t)\right\|^2 + 2C_x \mathbb{E}\left\|P_i^*(t)-P(t)\right\|^2 \nonumber\\
&\quad \le 2(C_P + C_x)\sigma. \label{dc41i}
\end{align}

Let $C_0 = 2(C_P + C_x)$ and $C_{\max} = \max\{C_1, \ldots, C_N\}$. Substituting inequality (\ref{dc41i}) into (\ref{dc41f}), we get 
\begin{align}
\mathbb{E}\|\bar{u}^*(t)-u^*(t)\|^2 \le \sum_{i=1}^N C_i \cdot C_0\sigma \le N C_{\max} C_0 \sigma. \label{dc41j}
\end{align}

Together with the fact that $\sigma > 0$ can be chosen arbitrarily small, there exist positive integers $\bar{M}$ such that for all $m, \xi, \tau>\bar{M}-1$, the inequality $\mathbb{E}\|\bar{u}^*(t)-u^*(t)\|^2 < \sigma$ holds. This completes the proof. \hfill $\blacksquare$

\section{Simulation Example}
In this section, we consider a multi-area smart grid consisting of six load frequency control areas. In particular, the six interconnected control areas have heterogeneous dynamic model parameters as listed in Table 1, including the generator rotational inertia constant $J_i/(kg\cdot m^{2})$, damping coefficients $G_i/(MW/Hz)$, governor coefficient $W_i$, and turbine/governor time constants $T_{t,i}/(s)$, $T_{g,i}/(s)$. In particular, the initial frequency deviation for each area is set to 0.1 $Hz$, and other states are initialized to zero, i.e., $\hat{x}_{i0}=[0.1~~0~~0~~0]'$.
\begin{table}[h]
\centering
\caption{System parameters.}
\label{table1}
\begin{tabular}{lcccccc}
\toprule
Areas & 1 & 2 & 3 & 4 & 5 & 6 \\
\toprule
$J_i$ & 10 & 11 & 10 & 12 & 10 & 11 \\
$G_i$ & 1.2 & 1.1 & 1.0 & 1.1 & 1.2 & 1.0 \\
$T_{t,i}$ & 0.31 & 0.30 & 0.32 & 0.30 & 0.30 & 0.34 \\
$T_{g,i}$ & 0.08 & 0.085 & 0.075 & 0.08 & 0.082 & 0.083 \\
$W_i$ & 2.4 & 2.45 & 2.5 & 2.3 & 2.35 & 2.4 \\
\bottomrule
\end{tabular}
\end{table}
Furthermore, considering the uncertainties in actual power systems, the initial state covariance for each area is selected as $R_{xi}=0.01I, i=1,\cdots,6$ to reflect the measurement uncertainty of the initial state values, the process noise covariance is selected as $R_{wi} = 0.01I$, which characterizes the effects arising from external load disturbance and modeling uncertainty, and the measurement noise covariance is selected as $R_{vi} = 0.01I$, which characterizes the effects arising from frequency and tie-line power signals. The communication topology and the physical interconnection lines between areas are shown in Fig.~\ref{fig3}, where interconnection line coefficients $T_{ij}$ between area $i$ and area $j$ are both 2 $p.u./Hz$.
\begin{figure}[!t]
\centering
  \includegraphics[width=6cm,height=3.5cm]{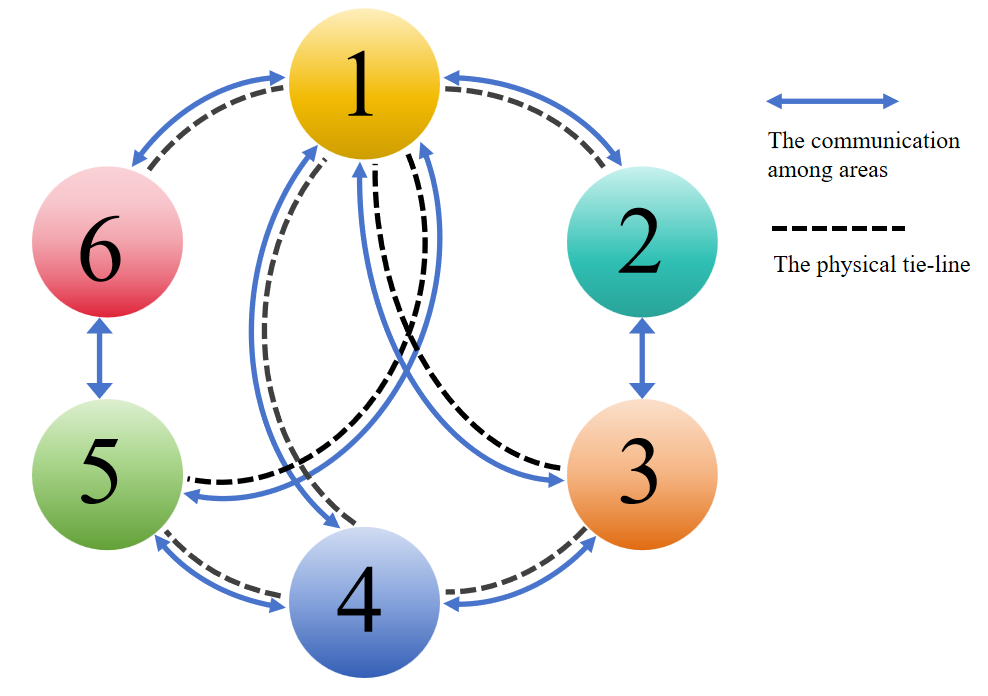}\\
  \caption{Communication topology and the physical interconnection lines between areas.}\label{fig3}
\end{figure}

To minimize control power consumption and frequency deviation over the finite time horizon $[0,T]$, the weighting matrices in (\ref{dc3}) are selected as follows:
\begin{eqnarray*}
&Q = I, R=I, T=6.
\end{eqnarray*}
In particular, $Q_i$ of each area can be expressed as $Q_1 = 1.4I,~ Q_2 = 1.2I, ~Q_3 = 0.8I,~Q_4 = 0.6I,~Q_5 = 1.1I, ~Q_6 = 0.9I$.

Based on Fig. 3 and the distributed iterative algorithms proposed in Section III, each control area $i$ computes its local optimal power regulation command by executing the following steps:
\begin{itemize}
\item \textbf{Step 1:} Computes the control Riccati matrix $P_i^m(t)$ by executing the distributed iterative algorithm (\ref{dc20})-(\ref{dc18}).
\item \textbf{Step 2:} Computes the estimation Riccati matrix $\Sigma_i^m(t)$ by executing the distributed iterative algorithm (\ref{dc22})-(\ref{dc19}).
\item \textbf{Step 3:} Computes the distributed state estimate $\bar{x}_{i,\tau}^*(t)$ by executing the iterative algorithm (\ref{dc33}) and (\ref{dc36}).
\end{itemize}

By running the implementation steps above with $\alpha_r = \alpha_\xi = \alpha_\tau = 1/k$, $\delta=10$, $\sigma=10^{-3}$, $m=30$, $r=\xi=\tau=800$, we obtain $P_i^*(t)$, $\Sigma_i^*(t)$, $\bar{x}_i^*(t)$ and $u_i^*(t)$, whose corresponding trajectories are shown in Figs.~\ref{fig4}-\ref{fig7}, respectively. It is evident that the distributed solutions from all areas converge to the centralized solutions over the finite horizon $[0,T]$.
\begin{figure}[h]
\centering
  \includegraphics[width=7.5cm,height=5cm]{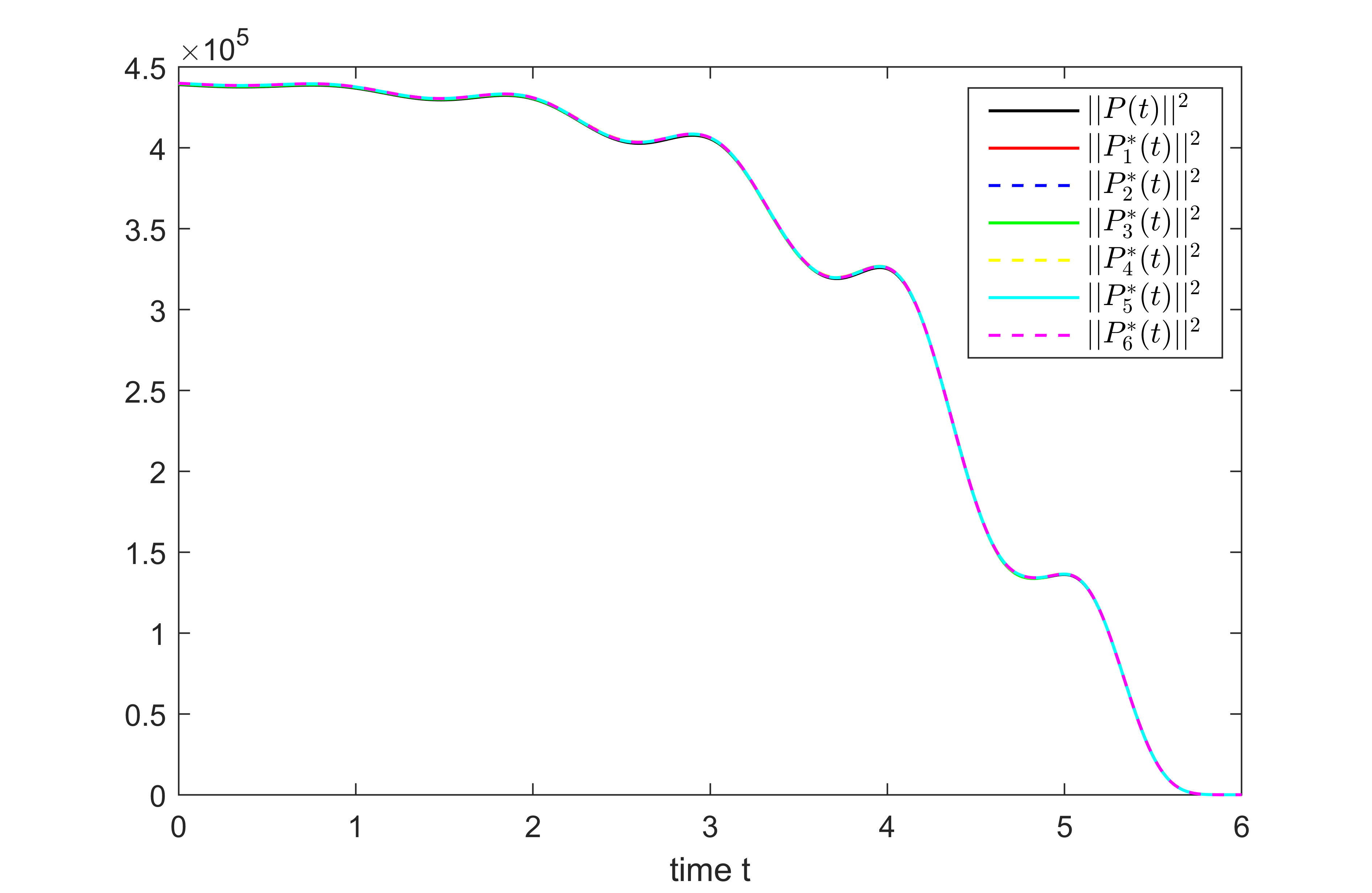}\\
  \caption{The trajectories of $\|P_{i}^{*}(t)\|^2$.}\label{fig4}
\end{figure}
%
\begin{figure}[!t]
\centering
  \includegraphics[width=7.5cm,height=5cm]{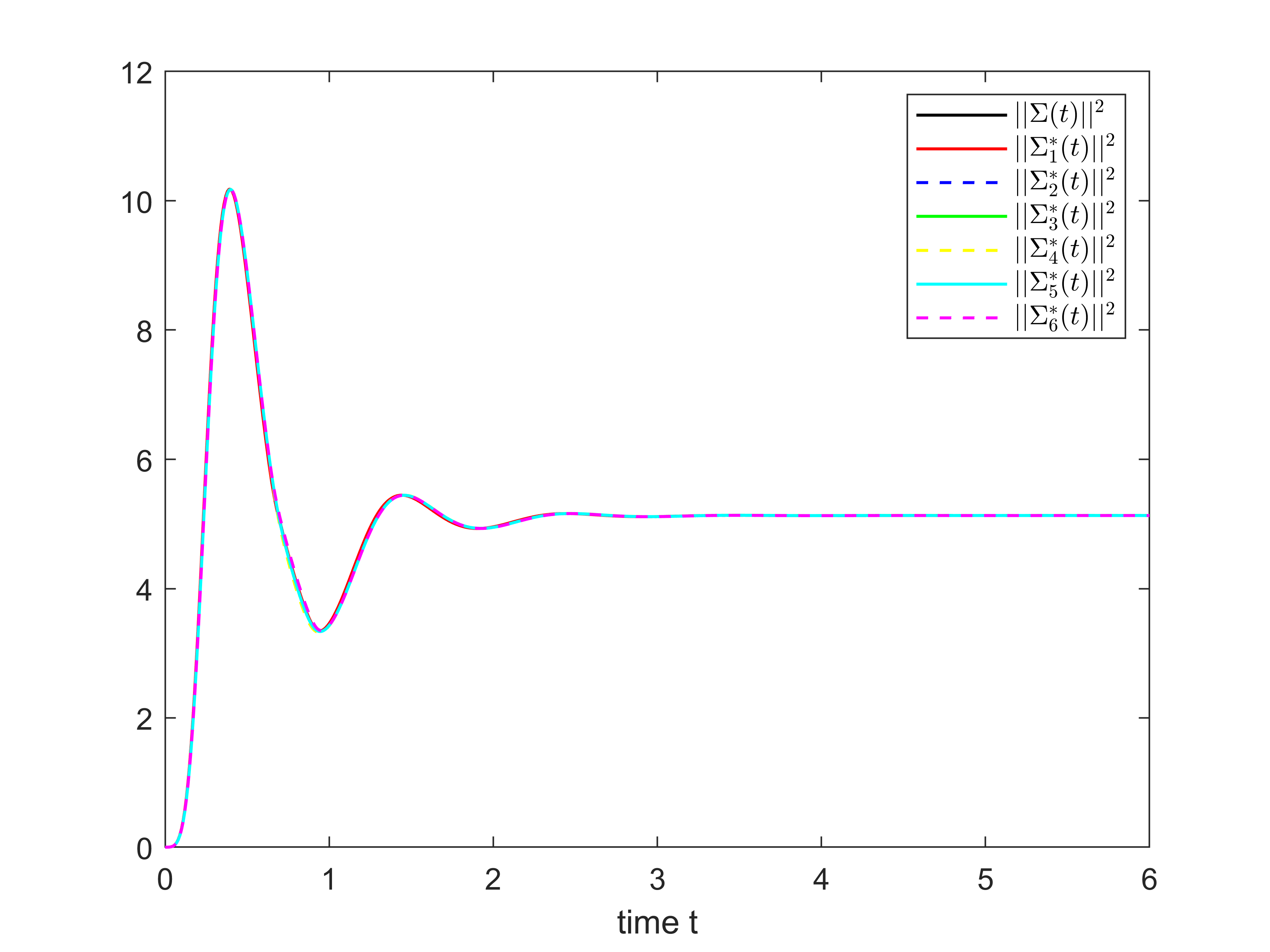}\\
  \caption{The trajectories of $\|\Sigma_{i}^{*}(t)\|^2$.}\label{fig5}
\end{figure}

\begin{figure}[!t]
\centering
  \includegraphics[width=7.5cm,height=5cm]{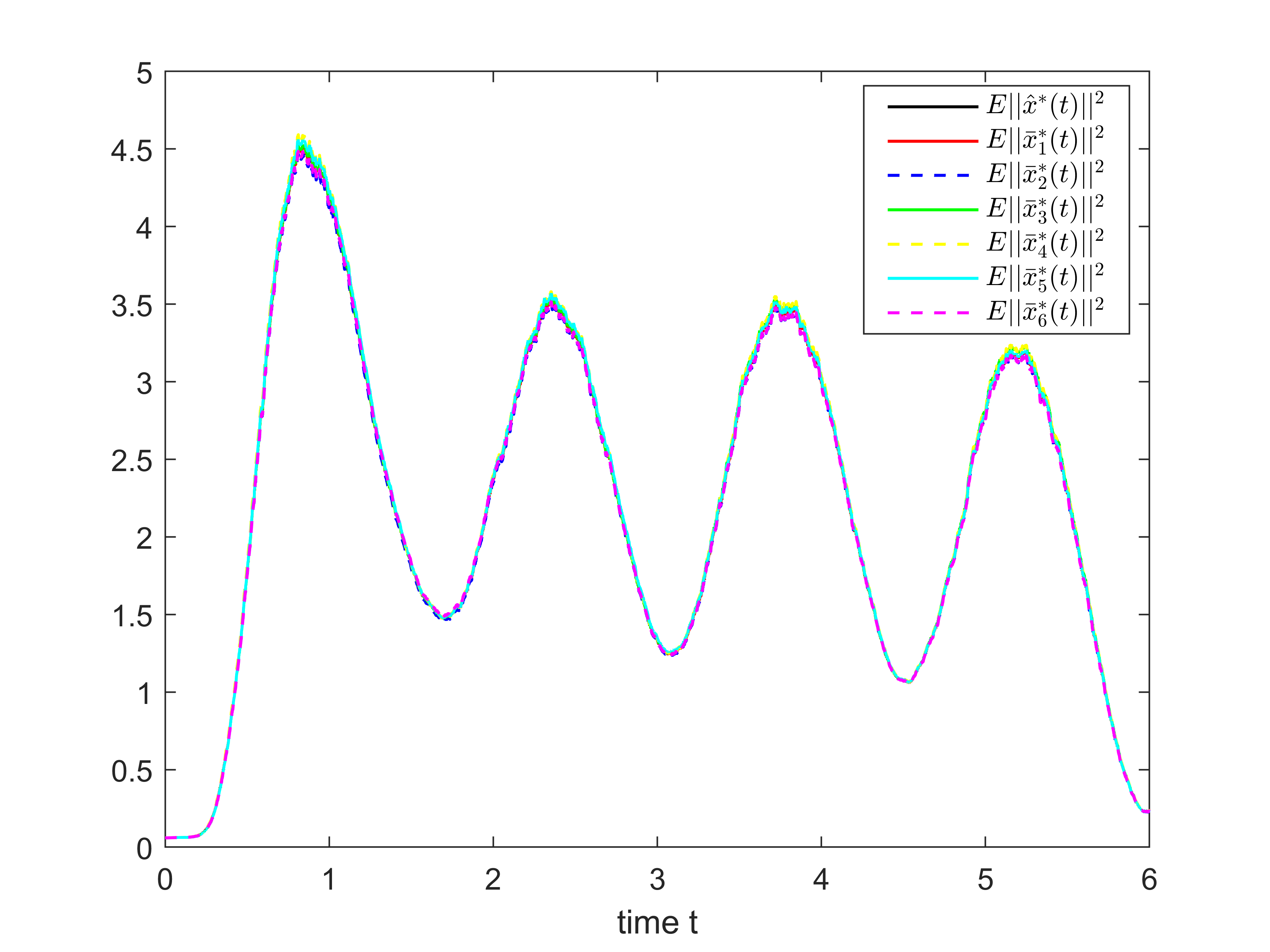}\\
  \caption{The trajectories of $E\|\bar{x}_{i}^{*}(t)\|^2$.}\label{fig6}
\end{figure}
\begin{figure}[!t]
\centering
  \includegraphics[width=7.5cm,height=5cm]{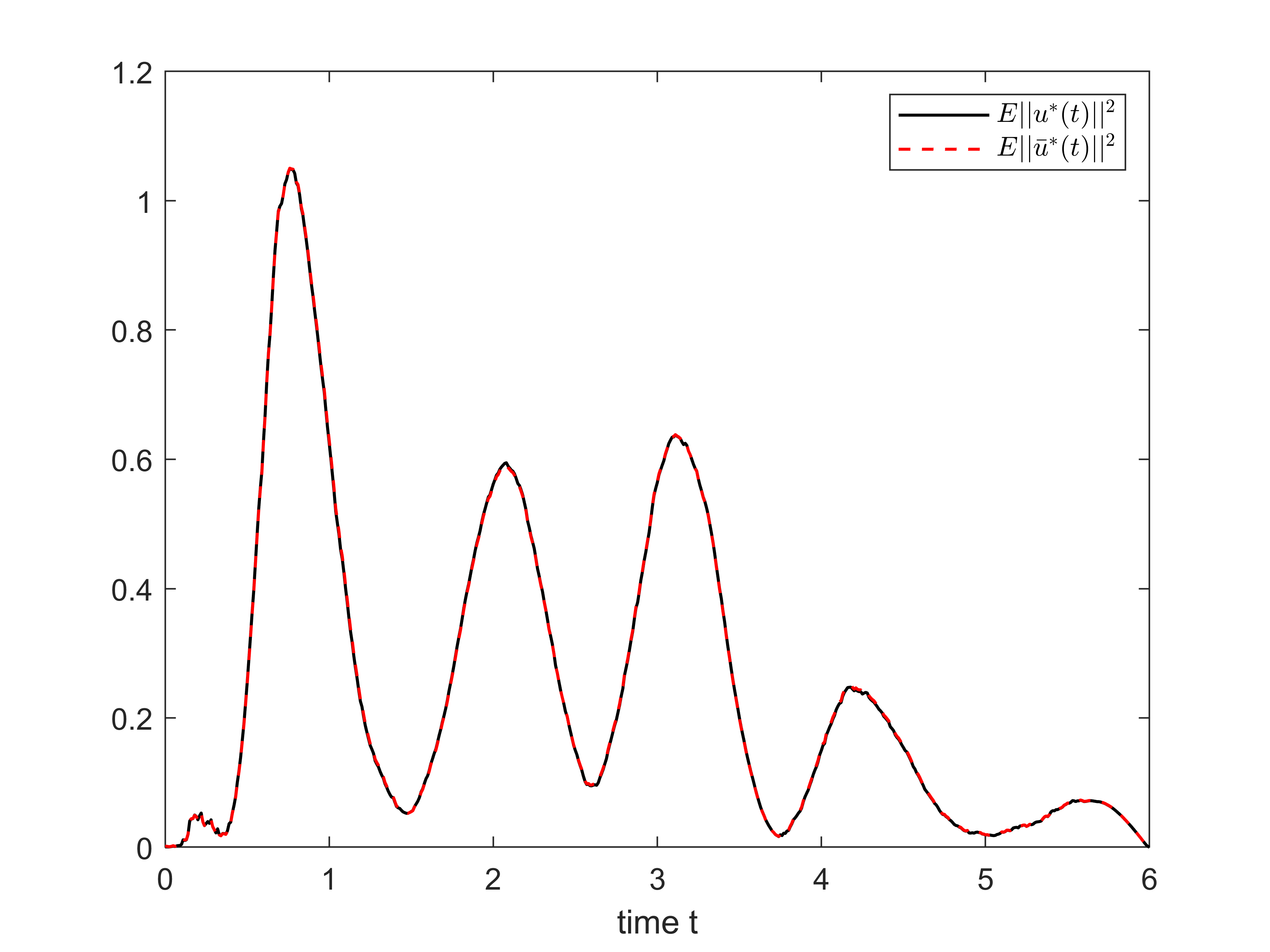}\\
  \caption{The trajectories of $E\|\bar{u}^{*}(t)\|^2$.}\label{fig7}
\end{figure}

Moreover, utilizing the obtained distributed control commands $u_i^*(t)$, the trajectories of $E\|\Delta f_i(t)\|^2$ and $E\|ACE_i(t)\|^2$ are shown in Figs.~\ref{fig8} and \ref{fig9}, respectively. It can be observed that both frequency deviation and area control error are reduced over the finite horizon $[0,T]$.
\begin{figure}[!t]
\centering
  \includegraphics[width=7.5cm,height=5cm]{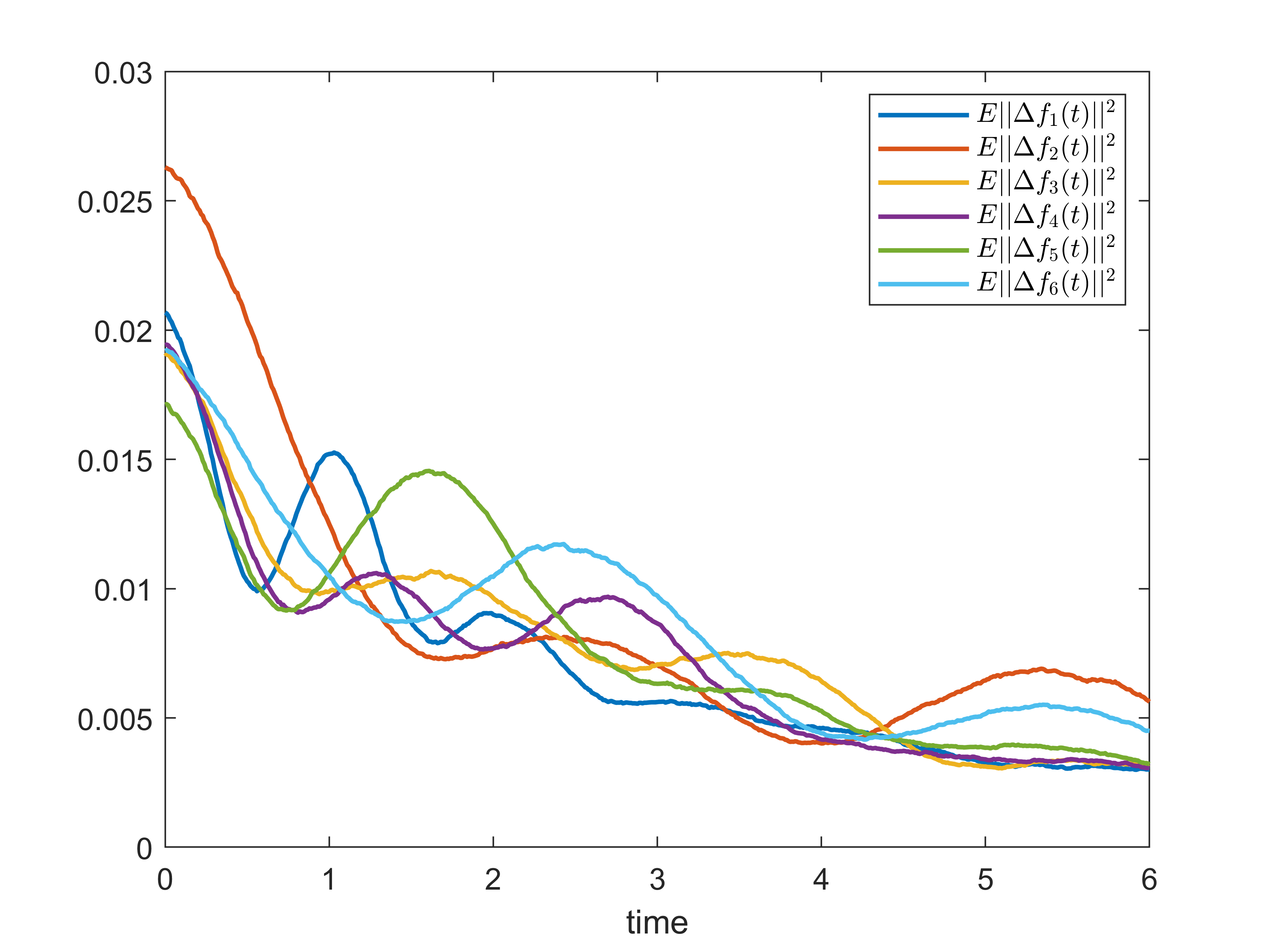}\\
  \caption{The trajectories of $E\|\Delta f_i(t)\|^2$.}\label{fig8}
\end{figure}
\begin{figure}[!t]
\centering
  \includegraphics[width=7.5cm,height=5cm]{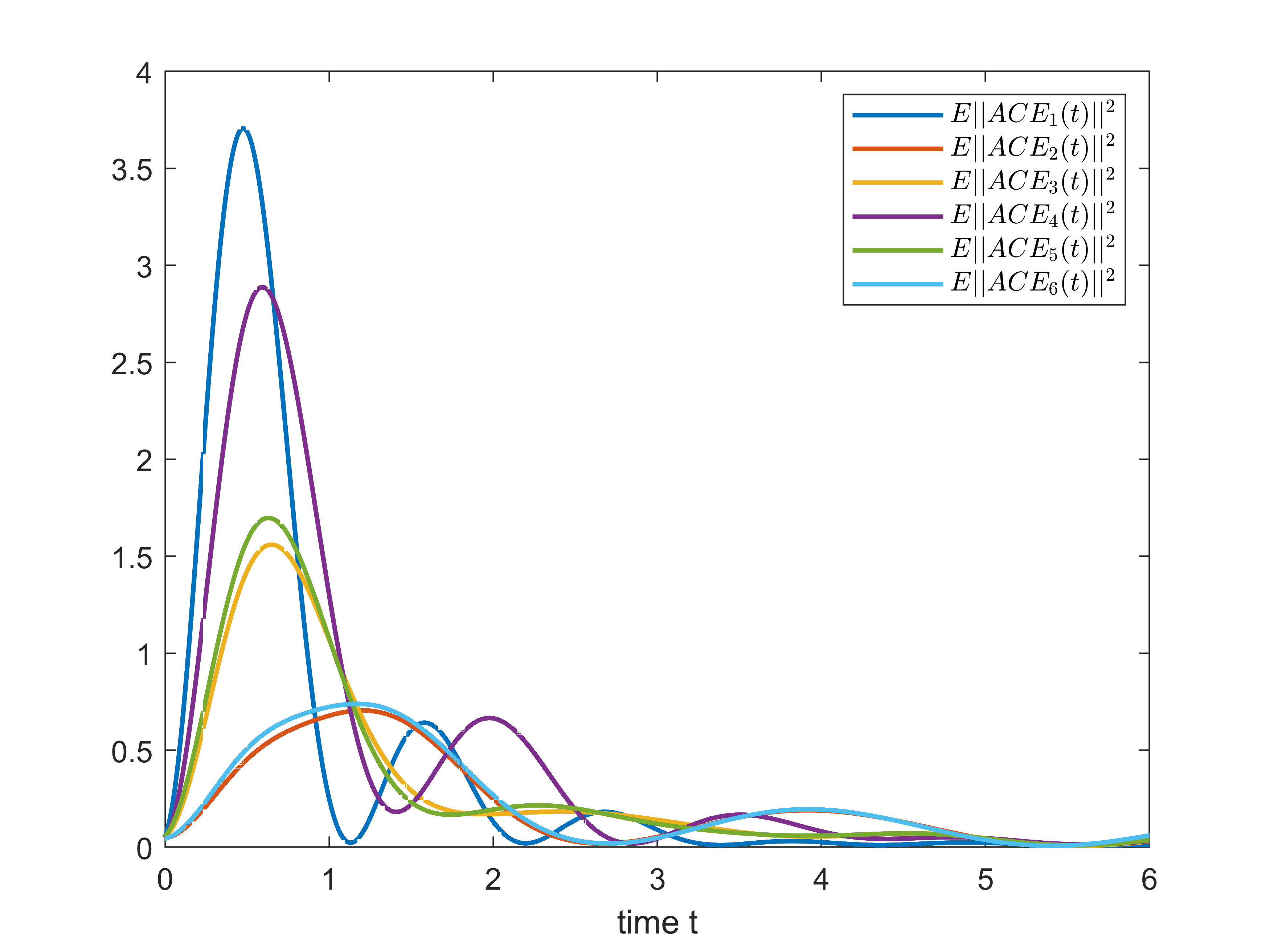}\\
  \caption{The trajectories of $E\|ACE_i(t)\|^2$.}\label{fig9}
\end{figure}

As a comparison, we also conduct simulations between the proposed distributed algorithm, the centralized optimal controller, and the distributed LFC method in [9] under identical initial conditions and system parameters. The quantitative results are summarized in Table~\ref{table2}. It can be observed that the performance index $J$ achieved by the proposed method approximates that of the centralized optimal controller and is smaller than that of the controller in [9] under different noise scenarios.
\begin{table}[h]
\centering
\caption{Performance index comparison among different controllers.}
\label{table2}
\begin{tabular}{l|l|l|l|l}
\hline
Case & $(R_{wi}, R_{vi})$ & $J$ (Proposed)  & $J$ (Centralized)& $J$ ([9]) \\
\hline
Case 1 & (0.4, 0.5) & 24.9134 & 24.9388& 42.1107  \\
Case 2 & (0.1, 0.1) & 19.5172  & 19.4962  & 39.4491\\
Case 3 & (0.02, 0.02) & 18.8083  & 18.7932 & 34.1017 \\
Case 4 & (0.6, 0.8) & 26.8736  & 26.8861 & 44.8122 \\
Case 5 & (0.002, 0.003) & 17.0317  & 17.0821  & 32.7498\\
\hline
\end{tabular}
\end{table}

\section{Conclusions}

This paper investigates distributed load frequency control in a multi-area smart grid under information privacy constraints, where internal structural parameters and operational states of each area are not shared with non-neighboring areas. A distributed iterative algorithm is proposed to approximate the global optimal power regulation command via distributed computation of the control Riccati equation, the estimation Riccati equation, and the state estimation. Numerical simulation results demonstrate that the performance of the proposed method can approximate centralized optimal control and is better than traditional distributed strategies.


\begin{thebibliography}{0}

\bibitem{1}
V. P. Singh, N. Kishor, and P. Samuel,
{``Load frequency control with communication topology changes in smart grid,''}
{\em IEEE Trans. Ind. Informat.}, vol. 12, no. 5, pp. 1943-1952, 2016.

\bibitem{2}
M. Kouki, B. Marinescu, and F. Xavier,
{``Exhaustive modal analysis of large-scale interconnected power systems with high power electronics penetration,''}
{\em IEEE Trans. Power Syst.}, vol. 35, no. 4, pp. 2759-2768, 2020.

\bibitem{3}
D. K. Chaturvedi,
{``Load frequency control in power systems,''}
Springer, 2011.

\bibitem{4}
P. S. V. Sagar and K. S. Swarup,
{``Load frequency control in isolated micro-grids using centralized model predictive control,''}
in {\em Proc. IEEE Int. Conf. Power Electron., Drives Energy Syst.}, Trivandrum, India, 2016, pp. 1-6.

\bibitem{5}
S. Rajesh and M. Swarup,
{``Decentralized model predictive control for load-frequency control of interconnected power system,''}
{\em IEEE Trans. Power Syst.}, vol. 30, no. 1, pp. 1-9, 2014.

\bibitem{6}
C. J. Ramlal, A. Singh, S. Rocke, and M. Sutherland,
{``Decentralized fuzzy $H_{\infty}$-iterative learning LFC with time-varying communication delays and parametric uncertainties,''}
{\em IEEE Trans. Power Syst.}, vol. 34, no. 6, pp. 4718-4727, 2019.

\bibitem{7}
N. M. Dehkordi, N. Sadati, and M. Hamzeh,
{``Distributed robust finite-time secondary voltage and frequency control of islanded microgrids,''}
{\em IEEE Trans. Power Syst.}, vol. 32, no. 5, pp. 3648-3659, 2017.

\bibitem{8}
Z. Hu, K. Zhang, R. Su, and R. Wang,
{``Robust cooperative load frequency control for enhancing wind energy integration in multi-area power systems,''}
{\em IEEE Trans. Autom. Sci. Eng.}, vol. 22, pp. 1508-1518, 2025.

\bibitem{9}
V. P. Singh, N. Kishor, and P. Samuel,
{``Distributed multi-agent system-based load frequency control for multi-area power system in smart grid,''}
{\em IEEE Trans. Ind. Informat.}, vol. 64, no. 6, pp. 5151-5160, Jun. 2017.

\bibitem{10}
S. Liu and P. X. Liu,
{``Distributed model-based control and scheduling for load frequency regulation of smart grids over limited bandwidth networks,''}
{\em IEEE Trans. Ind. Informat.}, vol. 14, no. 5, pp. 1814-1823, 2018.

\bibitem{11}
D. Chowdhury and H. K. Khalil,
{``Dynamic consensus and extended high gain observers as a tool to achieve practical frequency synchronization in power systems under unknown time-varying power demand,''}
{\em Automatica}, vol. 131, p. 109753, 2021.

\bibitem{12}
E. Vlahakis, L. D. Dritsas, and G. D. Halikias,
{``Distributed LQR design for identical dynamically coupled systems: Application to load frequency control of multi-area power grid,''}
in {\em Proc. IEEE Conf. Decis. Control}, Nice, France, 2019, pp. 4471-4476.

\bibitem{13}
L. Yang, T. Liu, Z. Tang, and D. J. Hill,
{``Distributed optimal generation and load-side control for frequency regulation in power systems,''}
{\em IEEE Trans. Autom. Control}, vol. 66, no. 6, pp. 2724-2731, 2021.

\bibitem{14}
H. Kim, M. Zhu, and J. Lian,
{``Distributed robust adaptive frequency control of power systems with dynamic loads,''}
{\em IEEE Trans. Autom. Control}, vol. 65, no. 11, pp. 4887-4894, 2020.

\bibitem{15}
E. Vlahakis, L. D. Dritsas, and G. D. Halikias,
{``Distributed LQR design for a class of large-scale multi-area power systems: Application to load frequency control of multi-area power grid,''}
{\em Energies}, vol. 12, no. 14, p. 2664, 2019.


\bibitem{16}
P. Ge, X. Dou, X. Quan, Q. Hu, Z. Wu, and W. Gu,
{``Extended-state-observer-based distributed robust secondary voltage and frequency control for an autonomous microgrid,''}
{\em IEEE Trans. Sustain. Energy}, vol. 11, no. 1, pp. 195-205, 2020.

\bibitem{17}
M. Zhang, S. Dong, P. Shi, G. Chen, and X. Guan,
{``Distributed observer-based event-triggered load frequency control of multiarea power systems under cyber attacks,''}
{\em IEEE Trans. Autom. Sci. Eng.}, vol. 20, no. 4, pp. 2435-2444, 2023.

\bibitem{18}
H. Haes Alhelou, M. E. Hamedani Golshan, and N. D. Hatziargyriou,
{``Deterministic dynamic state estimation-based optimal LFC for interconnected power systems using unknown input observer,''}
{\em Inf. Control}, vol. 38, no. 1, pp. 21-50, 1978.

\bibitem{19}
S. Xu, D. Ye, G. Li, and D. Yang,
{``Globally stealthy attacks against distributed state estimation in smart grid,''}
{\em IEEE Trans. Autom. Sci. Eng.}, vol. 22, pp. 1353-1363, 2025.

\bibitem{20}
Z. Wang, J. Dai, H. Zhang, and J. Zhang,
{``Observer-based finite-time fuzzy load frequency control for multiarea nonlinear power systems under input delays and cyber attacks,''}
{\em IEEE Trans. Smart Grid}, vol. 16, no. 4, pp. 3336-3345, 2025.

\bibitem{21}
J. Xia, X. Guo, J. H. Park, G. Chen, and X. Xie,
{``Predictor-based load frequency control for large-scale networked control power systems,''}
{\em IEEE Trans. Power Syst.}, vol. 39, no. 5, pp. 6263--6276, 2024.

\bibitem{22}
K. Liao and Y. Xu,
{``A robust load frequency control scheme for power systems based on second-order sliding mode and extended disturbance observer,''}
{\em IEEE Trans. Ind. Informat.}, vol. 14, no. 7, pp. 3076-3086, 2018.

\bibitem{23}
K. Zhou, J. C. Doyle, and K. Glover,
{``Robust and optimal control,''}
Prentice Hall, 1996.

\bibitem{24}
W. Yang, Z. Zhang, and J. Xu,
{``Distributed solving of linear quadratic optimal controller with terminal state constraint,''}
arXiv:2504.05631, 2025.

\end{thebibliography}
 \end{document}